\documentclass{my-aims}

\usepackage{amsmath}
\usepackage{paralist}
\usepackage{graphics}
\usepackage{epsfig}
\usepackage{graphicx}
\usepackage{epstopdf}
\usepackage[colorlinks=true]{hyperref}
\hypersetup{urlcolor=blue, citecolor=red}

\usepackage{algorithm}

\def\ppages{X--XX}

\newtheorem{theorem}{Theorem}[section]

\newtheorem{proposition}{Proposition}

\theoremstyle{definition}

\usepackage{xcolor}

\def\rr{\mbox{$\mathbb{R}$}}

\title[Optimal control of an HIV model]{%
Optimal control of an HIV model
with a trilinear antibody growth function}

\author[K. Allali, S. Harroudi and D. F. M. Torres]{}



\email{allali@hotmail.com}
\email{sanaa.harroudi@gmail.com}
\email{delfim@ua.pt}

\thanks{$^*$ Corresponding author: Delfim F. M. Torres}

\allowdisplaybreaks

\begin{document}

\maketitle

\centerline{\scshape Karam Allali and Sanaa Harroudi}
\medskip
{\footnotesize
	\centerline{Laboratory of Mathematics and Applications}
	\centerline{Faculty of Sciences and Technologies}
	\centerline{University Hassan II of Casablanca}
	\centerline{P. O. Box 146, Mohammedia, Morocco}
}

\medskip

\centerline{\scshape Delfim F. M. Torres$^*$}
\medskip
{\footnotesize
	\centerline{Center for Research and Development in Mathematics and Applications (CIDMA)}
	\centerline{Department of Mathematics, University of Aveiro, 3810-193 Aveiro, Portugal}
}

\bigskip


\begin{abstract}
We propose and study a new mathematical model of the human
immunodeficiency virus (HIV). The main
novelty is to consider that the antibody growth depends not only
on the virus and on the antibodies concentration but also on the
uninfected cells concentration. The model consists of five
nonlinear differential equations describing the
evolution of the uninfected cells, the infected ones, the free
viruses, and the adaptive immunity. The  adaptive immune response is
represented by the cytotoxic T-lymphocytes (CTL) cells and the
antibodies with the growth function supposed to be trilinear.
The model includes two kinds of treatments. The objective of the
first one is to reduce the number of infected cells, while the aim
of the second is to block free viruses. Firstly, the positivity
and the boundedness of solutions are established. After that, the
local stability of the disease free steady state and the infection
steady states are characterized. Next, an optimal control problem is posed
and investigated. Finally, numerical simulations are performed in
order to show the behavior of solutions and the effectiveness of the
two incorporated treatments via an efficient optimal control strategy.
\end{abstract}


\section{Introduction}
\label{sec:1}

Human immunodeficiency virus (HIV) remains a worldwide health
problem that can cause the well known acquired immunodeficiency
syndrome (AIDS). Once it invades the body, HIV virus begins to
destruct the vast majority of CD4$^+$  T cells, often referred to as
``helper'' cells. These cells can be considered the command centers of
the immune system \cite{crs}. The immune system is represented by
the cytotoxic T lymphocytes (CTLs) and antibodies respond to their
message by attacking and killing the infected cells and HIV virus.
In the last decades, many mathematical models have been developed to
better describe and understand the dynamics of the HIV disease, e.g.,
\cite{MR3808514,Nov,Per1,MyID:359,MR3703337}. An HIV model with
adaptive immune response, two saturated rates, and therapy, is
studied in \cite{all}, showing that the goal of immunity response is
controlling the load of HIV viruses. Mathematical models of HIV and
tuberculosis coinfection have been carried out in
\cite{MR3804169,sil,MR3392642}. Models of HIV infection
using optimization techniques and optimal control in the study of
HIV have been investigated in \cite{roch,MR3721854,MR3714435}.
Recently, the same problem was tackled by introducing the HIV virus
dynamics into the system of equations in view of his importance in the
infection \cite{har}. Here, we continue the investigation of such
kind of problems by introducing antibodies immune response. Similar
models can be found in \cite{Nowak:May:2000}. Wodarz wrote an entire monograph
reviewing different models for CD8 cells \cite{MR2273003}. In 2013,
De Boer and Perelson have reviewed the existing literature on T-cell models
\cite{MR3046079}. For previous HIV modeling studies using optimal control theory
to determine optimal treatment protocols, we refer the reader to
\cite{har,MR3804169,roch,MR3721854,MR3872461,MR3392642,MR3714435}
and references therein. Finally, it should be mentioned that there is
abundant data on viral and T cell kinetics during HIV
and simian immune deficiency (SIV) infection
and the effects of therapy. For an example of an experimental
study that quantifies the effects of therapy, see, e.g.,
\cite{Davenport:2004}, where data on SIV and CTL
cell kinetics during primary monkey infection is provided.
For similar compartmental models in different contexts
see \cite{MyID:432,MR3831969}.
The main novelty here is to consider that the antibody growth
depends not only on the virus and on the antibodies concentration
but also on the uninfected cells concentration. That was never
investigated before, from a mathematical point of view,
but it is very important since the role of the immune
response to HIV infection has been recently recognized
by the medical literature to be of a great value. 
Indeed, it is now well known that the CTL immune response 
grows depending on the infected cells. 
This growth also depends on the number of CTL cells themselves. Moreover, 
the antibody immune response grows depending on the virus proliferation 
and this growth also depends on the number  of viruses. 
Because the growth of the immune system cells depends on the number 
of healthy target cells CD4+ T cells, hence the trilinear term 
to describe the growth of the immune responses \cite{crs,32,crs2}. 
The goal of HIV virus is to destruct CD4+ T cells, 
often named ``messengers'' or the command
centers of the immune system. Once the virus invades the body,
these cells give a signal to the immune system. The immune system
is represented by CTL and antibodies that respond to this message
and set out to eliminate the infection by killing infected cells and free virus.
To include into the model the antibodies participation in controlling the infection
is thus essential. The mathematical model we propose is the following one:
\begin{equation}
\label{sy}
\left\{ \begin{array}{llll}
\vspace{0.1cm} \displaystyle \frac{dx}{dt} = \lambda -dx- \beta xv, \\
\vspace{0.1cm} \displaystyle \frac{dy}{dt} = \beta xv- ay- pyz,\\
\vspace{0.1cm} \displaystyle \frac{dv}{dt} = aNy- \mu v- qvw , \\
\vspace{0.1cm} \displaystyle \frac{dz}{dt} = cxyz - hz,\\
\vspace{0.1cm} \displaystyle \frac{dw}{dt} = gxvw- \alpha w,\\
\end{array}
\right.
\end{equation}
with given initial conditions
\begin{equation}
\label{eq:IC}
x(0) = x_{0}, \quad y(0) = y_{0}, \quad v(0) = v_{0},
\quad z(0)= z_{0}, \quad w(0)= w_{0}.
\end{equation}
In this model, $x(t)$, $y(t)$, $v(t)$, $z(t)$, and $w(t)$, denote the
concentrations of uninfected cells, infected cells, HIV virus, CTL
cells, and antibodies at time $t$, respectively. The healthy CD$4^{+}$ T cells
$(x)$ grow at a rate $\lambda$, die at a rate $d$, and become
infected by the virus at a rate $\beta x v$. Infected cells $(y)$,
die at a rate $a$ and are killed by the CTLs response at a rate $p$.
Free virus $(v)$ is produced by the infected cells at a rate $aN$,
die at a rate $\mu$, and decay in the presence of antibodies at a
rate $q$, where $N$ is the number of free virus produced by each
actively infected cell during its life time. CTLs $(z)$ expand, in
response to viral antigen derived from infected cells, at a rate $c$
and decay in the absence of antigenic stimulation at a rate $h$.
Finally, antibodies $(w)$ develop in response to free virus
at a rate $g$ and decay at a rate $\alpha$. It is worthy to note 
that all the model rates are assumed to be nonnegative.

The paper is organized as follows. Section~\ref{sec:2} is devoted to
the existence, positivity, and boundedness of solutions. The analysis
of the model is carried out in Section~\ref{sec:3}. In Section~\ref{sec:4},
an HIV optimal control problem is posed and solved. Then, in Section~\ref{sec:5},
the results are illustrated through numerical simulations.
We finish with Section~\ref{sec:6} of conclusions.


\section{Well-posedness of solutions}
\label{sec:2}

For problems dealing with cell population
evolution, the cell densities should remain non-negative and
bounded. In this section, we establish the positivity and
boundedness of solutions of the model \eqref{sy}. First of all, for
biological reasons, the parameters $x_0$, $y_0$, $v_0$, $z_0$, and
$w_0$, must be larger than or equal to zero. Hence, we have the
following result.

\begin{proposition}
\label{prop1}
The solutions of the problem \eqref{sy} exist. Moreover, they are
bounded and nonnegative for all $t > 0$.
\end{proposition}

\begin{proof}
First, we show that the nonnegative orthant
$\rr_{+}^{5}=\{(x,y,v,z,w)\in \rr^{5}: x\geq0, y\geq0, v\geq0,
z\geq0 \ \text{and}\ w\geq0\}$ is positively invariant. Indeed, for
$\left(x(t), y(t), v(t), z(t), w(t)\right) \in \rr_{+}^{5}$, we have:
$\dot x\mid _{x=0}= \lambda \geq0$, $\dot y\mid _{y=0}= \beta xv\geq0$,
$\dot v\mid _{v=0}= aNy\geq0$, $\dot z\mid_{z=0}=0\geq0$,
and $\dot w\mid _{w=0}=0\geq0$. Therefore, all
solutions initiating in $\rr_{+}^{5}$ are positive.
Next, we prove that these solutions remain bounded. Remark
that, by adding the two first equations in \eqref{sy}, we have
$\dot{x_1}= \lambda-dx-ay-pyz$, thus
\begin{equation*}
x_1(t)\leq x_1(0)e^{-\delta t}+\dfrac{\lambda}{\delta}(1-e^{-\delta t}),
\end{equation*}
where $x_{1}(t)=x(t)+y(t)$ and $\delta=\min(d;a)$.
Since $0\leq e^{-\delta t}\leq 1$ and $ 1-e^{-\delta t}\leq 1$,
we deduce that $x_{1}(t)\leq x_{1}(0)+\dfrac {\lambda}{\delta}$.
Therefore, $x$ and $y$ are bounded.
From the equation $\dot{v}= aNy -\mu v-qvw$, we have
\begin{equation*}
v(t)\leq v(0)e^{-\mu t}+ aN\int^{t}_{0}y(\xi)e^{(\xi - t)\mu} d\xi.
\end{equation*}
Then,
$$
v(t)\leq v(0)+ \dfrac{aN}{\mu}\left\|y\right\|_{\infty}(1-e^{-\mu t}).
$$
Since $1-e^{-\mu t}\leq 1$, we have $v(t)\leq v(0)+\dfrac
{aN}{\mu}\left\|y\right\|_{\infty}$. Thus, $v$ is bounded.
Now, we prove the boundedness of $z$. From the fourth equation of
\eqref{sy}, we have
$$
\dot{z}(t) + hz(t) = cx(t)y(t)z(t).
$$
Moreover, from the second equation of \eqref{sy}, it follows that
$$
\dot{z}(t) + hz(t) = \frac{c}{p}x(t)\left(\beta x(t)v(t)- ay(t)-
\dot{y}(t) \right).
$$
By integrating over time, we have
$$
z(t) = z(0) e^{-ht} + \int_{0}^t \frac{c}{p}x(s)\left(  \beta
x(s)v(s)- ay(s) - \dot{y}(s)\right) e^{h(s-t)}ds.
$$
From the boundedness of $x$, $y$, and $v$, and by using integration
by parts, it follows the boundedness of $z$.
The two equations $\dot{v}(t)= aNy(t)-\mu v(t)-qv(t)w(t)$
and $\dot{w}(t)=gx(t)v(t)w(t)- \alpha w(t)$ imply
$$
\dot{w}(t)+ \alpha w(t)=gx(t)v(t)w(t)
=\dfrac{g}{q} x(t) \left(aNy(t)-\dot{v}(t)
-\mu v(t) \right).
$$
Then,
\begin{equation*}
w(t) = w(0) e^{- \alpha t} + \int_{0}^t
\frac{g}{q}x(s)\left( aNy(s)- \mu v(s)
- \dot{v}(s)\right) e^{ \alpha(s-t)}ds.
\end{equation*}
From the boundedness of $x$, $y$, and $v$, and by integration
by parts, it follows the boundedness of $w$.
\end{proof}


\section{Analysis of the model}
\label{sec:3}

In this section, we show that there exists a disease free equilibrium
point and four infection equilibrium points. Moreover, we study the
stability of these equilibrium points.


\subsection{Stability of the disease-free equilibrium}
\label{sec:3:1}

System \eqref{sy} has an infection-free equilibrium
$E_f=\left(\dfrac{ \lambda}{d},0,0,0,0\right)$,
corresponding to the maximal level of
healthy CD4$^{+}$ T-cells. In this case, the disease cannot invade
the cell population. By a simple calculation \cite{van-den-Driessche},
the basic reproduction number of \eqref{sy} is given by
\begin{equation*}
R_{0}=\dfrac{ \lambda N \beta}{d\mu }.
\end{equation*}
At any arbitrary point, the Jacobian matrix of the system
\eqref{sy} is given by
\begin{equation*}
J=\left(
\begin{array}{clcrc}
-d- \beta v & 0 & - \beta x & 0 & 0  \\
\beta v& -a-pz & \beta x & -py & 0 \\
0 & aN & -\mu-qw &0 & -qv\\
cyz & cxz & 0 & cxy-h & 0\\
gvw & 0 & gxw & 0 & gxv- \alpha
\end{array}
\right).
\end{equation*}

\begin{proposition}
\mbox {}
\begin{enumerate}
\item The disease-free equilibrium, $E_f $, is locally
asymptotically stable for $R_0 < 1$.
\item The disease-free equilibrium, $E_f $, is unstable for $R_0 > 1$.
\end{enumerate}
\end{proposition}

\begin{proof}
At the disease-free equilibrium, $E_f $, the Jacobian matrix  is
given as follows:
\begin{equation*}
J_{E_f}=\left(
\begin{array}{clcrc}
-d & 0 & -\dfrac{\beta \lambda}{d} & 0&0  \\
0& -a & \dfrac{\beta \lambda}{d} & 0&0 \\
0 &aN & -\mu &0&0\\
0&0&0&-h&0\\
0 & 0&0&0&-\alpha
\end{array}
\right).
\end{equation*}
The characteristic polynomial of $J_{E_f}$ is
$$
P_{E_f}(\xi)
=(\xi+d)(\xi+ \alpha)(\xi+h)[\xi^2+(a+\mu)\xi+a\mu(1-R_0)]
$$
and the eigenvalues of the matrix $J_{E_f}$ are
\begin{equation}
\label{PE}
\begin{aligned}
\xi_1&=-d,& \\ \nonumber
\xi_2&=- \alpha,&\\
\xi_3&=-h,&\\
\xi_4&=\displaystyle\dfrac{-(a+\mu)-\sqrt{(a+\mu)^2-4a\mu(1-R_0)}}{2},&\\
\xi_5&=\displaystyle\dfrac{-(a+\mu)+\sqrt{(a+\mu)^2-4a\mu(1-R_0)}}{2}.&
\end{aligned}
\end{equation}
It is clear that $\xi_1$, $\xi_2$, $\xi_3$ and $\xi_4$ are negative.
Moreover, $\xi_5$ is negative when $R_0 < 1$, which means that $E_f$
is locally asymptotically stable.
\end{proof}


\subsection{Infection steady states}
\label{sec:3:2}

We now focus on the existence and stability of the
infection steady states. All these steady states exist when the
basic reproduction number exceeds the unity and the disease
invasion is always possible. In fact, it is easily verified
that the system \eqref{sy} has four of them:
$$
E_1=\left(\dfrac{\mu}{ \beta N},
\frac{d \mu (R_{0}-1)}{aN \beta},\frac{d(R_{0}-1)}{\beta},0,0\right),
$$
$$
E_2=\left(\frac{ \lambda \mu c- aN \beta h}{d \mu c},
\frac{dh \mu}{\lambda \mu c- aN \beta h},
\dfrac{aNdh}{\lambda \mu c- aN \beta h},\dfrac{a}{p}(R^{CTL}-1),0\right),
$$
$$
E_3=\left(\frac{\lambda g- \alpha \beta }{dg}, \frac{\alpha \beta}{ag},
\frac{\alpha d}{\lambda g- \alpha \beta},0,\frac{\mu}{q}(R^{W}-1)\right),
$$
$$
E_4=\left(\frac{\lambda g- \alpha \beta }{dg},\dfrac{hdg}{c(\lambda g- \alpha \beta)},
\dfrac{\alpha d}{(\lambda g- \alpha \beta)},\frac{a}{p}(R^{CTL,W}_{2}-1),
\frac{\mu}{q}(R^{CTL,W}_{1}-1)\right).
$$
Here the endemic equilibrium point $E_1$ represents the equilibrium
case in the absence of the adaptive immune response (CTLs and
antibody responses). The endemic equilibria points $E_2$ and $E_3$
represent the equilibrium case in the presence of only one kind of
the adaptive immune response, antibody response and CTL response,
respectively, while the last endemic equilibrium point $E_4$
represents the equilibrium case of chronic HIV infection with the
presence of both kinds of adaptive immune response, CTLs and antibody.
In order to study the local stability of the points $E_1$,
$E_2$, $E_3$ and $E_4$, we first define the following numbers:
$$
R^{CTL}= \frac{N \beta}{\mu}\left(\frac{\lambda \mu c
- aN\beta h}{d \mu c}\right),
$$
where $R^{CTL}$ represents the reproduction number
in presence of CTL immune response,
$$
R^{W}= \frac{N \beta(\lambda g- \alpha \beta)}{\mu dg},
$$
where $R^{W}$ represents the reproduction number
in presence of antibody immune response,
$$
R^{CTL,W}_{1}= \frac{aNhg}{ \alpha \mu c}
\quad \text{ and } \quad
R^{CTL,W}_{2}= \frac{\alpha \beta c(\lambda g- \alpha \beta)}{ahdg^{2}},
$$
where $R^{CTL,W}_{1}$ and $R^{CTL,W}_{2}$ represent the reproduction
number in presence of antibody immune response and CTL immune
response, respectively. For the first point $E_1$,
we have the following result.

\begin{proposition}
\begin{enumerate}
\item If $R_0<1$, then the point $E_1$ does not exist.
\item If $R_0=1$, then $E_1=E_f$.
\item If $R_0>1$, then $E_1$ is locally asymptotically stable for
$R^{W}<1$ and $R^{CTL}<1$. However, it is unstable
for $R^{W}>1$ or $R^{CTL}>1$.
\end{enumerate}
\end{proposition}

\begin{proof}
Let $\lambda \mu c- aN \beta h >0$. It is easy to
see that if $R_0 < 1$, then the point $E_1$ does not exist and if
$R_0 = 1$, then the two points $E_1$ and $E_f$ coincide.
If $R_0 > 1$, then the Jacobian matrix at $E_1$ is given by
\begin{equation*}
J_{E_1}=\left(
\begin{array}{clcrc}
-d-\beta v_1 & 0 & -\beta x_1 & 0&0  \\
 \beta v_1& -a & \beta x_1 &-py_{1}& 0 \\
0 & aN& -\mu &0&-qv_{1}\\
0 & 0 & 0 &cx_{1}y_{1}-h&0\\
0 & 0&0&0&gx_{1}v_{1}- \alpha
\end{array}
\right).
\end{equation*}
Its characteristic equation is
\begin{equation*}
(cx_{1}y_{1}-h-\xi)(gx_{1}v_{1}- \alpha - \xi)
(\xi^3+{\mathcal A}_{1}\xi^2+{\mathcal B}_{1}\xi+{\mathcal C}_{1})=0,
\end{equation*}
where
\begin{equation*}
\begin{aligned}
{\mathcal A}_{1}&=a+ \mu +dR_{0},& \\
{\mathcal B}_{1}&=ad+ \mu dR_{0}+ad(R_{0}-1),&\\
{\mathcal C}_{1}&=ad \mu(R_{0}-1).&\\
\end{aligned}
\end{equation*}
Direct calculations lead to
$$
gx_1v_1- \alpha = D_{1}(R^{W}-1)
\mbox{ ~and~ } cx_{1}y_{1}-h= D_{2}(R^{CTL}-1)
$$
with
$$
D_{1}=\frac{dg \mu}{N \beta^{2}} \mbox{ ~and~ }
D_{2}=\frac{dc \mu^{2}}{aN^{2} \beta^{2}}.
$$
The sign of the eigenvalue $D_{1}(R^{W}-1)$ is negative if
$R^{W}<1$, zero if $R^{W}=1$, and positive if $R^{W}>1$.
The sign of the eigenvalue $D_{2}(R^{CTL}-1)$ is negative
if $R^{CTL}<1$, zero if $R^{CTL}=1$, and positive
if $R^{CTL}>1$. On the other hand, we have
${\mathcal A}_{1}>0$ and ${\mathcal A}_{1}{\mathcal B}_{1}-{\mathcal C}_{1}>0$
(as $R_0>1$). From the Routh--Hurwitz theorem \cite{Gradshteyn},
the other eigenvalues of the above matrix have negative real parts.
Consequently, $E_1$ is unstable when $R^{W} > 1$ or $R^{CTL} > 1$
and locally asymptotically stable when $R_0 > 1$,
$R^{W} < 1$, and $R^{CTL} < 1$.
\end{proof}

For the second endemic-equilibrium point $E_2$, we have the
following result.

\begin{proposition}
\begin{enumerate}
\item If $ R^{CTL} < 1 $ , then the point $ E_{2} $ does not exists
and $ E_{2} = E_{1} $ when $R^{CTL}= 1 $.
\item If $ R^{CTL} > 1 $ and $R^{CTL,W}_{1}\leq 1$, then
$E_{2}$ is locally asymptotically stable.
\item If $ R^{CTL} > 1 $ and $R^{CTL,W}_{1}>1$,
then $ E_{2} $ is unstable.
\end{enumerate}
\end{proposition}

\begin{proof}
Let $\lambda \mu c- aN \beta h >0$. If $ R^{CTL} < 1$, then the
point $E_{2}$ does not exists and $E_{2} = E_{1}$ when
$R^{CTL} = 1$. We assume that $R^{CTL} > 1$.
The Jacobian matrix of $E_{2}$ is given as follows:
\begin{equation*}
J_{E_{2}}=\left(
\begin{array}{clcrc}
-d-\beta v_2& 0 & -\beta x_{2} & 0 &0 \\
\beta v_2& -a-pz_{2} & \beta x_{2}& -py_{2}&0 \\
0 &aN & -\mu &0&-qv_{2}\\
cy_{2}z_{2} & cx_{2}z_{2} &0& cx_{2}y_{2}-h&0\\
0 & 0&0&0&gx_{2}v_{2}- \alpha
\end{array}
\right).
\end{equation*}
The characteristic equation of the system \eqref{sy} at the point
$E_{2}$ is given by
$$
(gx_{2}v_{2}- \alpha- \xi)
(\xi^4+{\mathcal A}_{2} \xi^3+{\mathcal B}_{2}\xi^2+{\mathcal C}_{2}\xi+{\mathcal D}_{2})=0,
$$
where
\begin{equation*}
\begin{aligned}
{\mathcal A}_{2}
&= d+a+ \mu +\beta v_{2}+pz_{2},&\\
{\mathcal B}_{2}
&= (d+ \beta v_{2})(a+ \mu)+a \mu
+ pz_{2}(d+ \mu +h + \beta v_{2})- aN \beta x_{2},&\\
{\mathcal C}_{2}
&=a \mu (d+ \beta v_{2})+ pz_{2}( \mu d+hd + \mu h+ \mu \beta v_{2}
+h \beta v_{2})- aN \beta d x_{2},&\\
{\mathcal D}_{2}
&= pz_{2}( \mu hd+ \mu h \beta v_{2}- aN \beta hy_{2}).&\\
\end{aligned}
\end{equation*}
Simple calculations lead to
$$
gx_{2}v_{2}- \alpha= \alpha \big(R^{CTL,W}_{1}-1\big).
$$
Then, $gx_{2}v_{2}- \alpha= \alpha \big(R^{CTL,W}_{1}-1\big)$ is an
eigenvalue of $J_{E_{2}}$. The sign of this eigenvalue is negative
if $ R^{CTL,W}_{1}<1$, null when $ R^{CTL,W}_{1}=1$, and positive if
$R^{CTL,W}_{1}>1$. On the other hand, from the Routh--Hurwitz
theorem, the other eigenvalues of the above matrix have negative
real part when $R^{CTL} > 1$. Consequently, $ E_{2}$ is unstable
when $R^{CTL} > 1$ and $ R^{CTL,W}_{1}>1$ and locally
asymptotically stable when $R^{CTL}> 1$ and $ R^{CTL,W}_{1}<1$.
\end{proof}

For the third endemic-equilibrium point $E_3$,
the following result holds.

\begin{proposition}
\mbox {}
\begin{enumerate}
\item If $R^{W}<1$, then the point $E_3$ does not exist and
$E_3=E_1$ when $R^{W}=1$. \item If $R^{W}>1$, then $E_3$ is locally
asymptotically stable for $R^{CTL,W}_{2} <1$ and unstable if
$R^{CTL,W}_{2}>1$.
\end{enumerate}
\end{proposition}

\begin{proof}
It is clear that when $R^{W} < 1$ the point $E_{3}$
does not exist and, if $R^{W} = 1$, then $E_{3} = E_{1}$. We assume
that $R^{W}> 1$. The Jacobian matrix of the system at point
$E_{3}$ is given by
\begin{equation*}
J_{E_{3}}=\left(
\begin{array}{clcrc}
-d- \beta v_{3}& 0 & -\beta x_{3} & 0 &0 \\
\beta v_{3}& -a & \beta x_{3}& -py_{3}&0 \\
0 & aN & -\mu-qw_{3} &0&-qv_{3}\\
0 & 0 & 0 &cx_{3}y_{3}-h&0\\
gv_{3}w_{3} & 0 & gx_{3}w_{3} & 0 & gx_{3}v_{3}- \alpha
\end{array}
\right).
\end{equation*}
The characteristic equation associated with $J_{E_3}$ is given by
\begin{equation*}
(cx_{3}y_{3}-h- \xi)(\xi^4+{\mathcal A}_{3}\xi^3
+{\mathcal B}_{3}\xi^2+{\mathcal C}_{3}\xi+{\mathcal D}_{3})=0,
\end{equation*}
where
\begin{equation*}
\begin{aligned}
{\mathcal A}_{3}&= a+d+ \mu + \beta v_{3} +qw_{3},& \\
{\mathcal B}_{3}&= (d+ \beta v_{3})(a+ \mu)+ a \mu
+(d+a+ \alpha+ \beta v_{3})qw_{3}-aN \beta x_{3},&\\
{\mathcal C}_{3}&= a \mu (d+ \beta v_{3})+(ad+ \alpha d
+ a \alpha+ a \beta v_{3})qw_{3}-aNd \beta x_{3} ,&\\
{\mathcal D}_{3}&= ad \alpha qw_{3}.&\\
\end{aligned}
\end{equation*}
Here $cx_{3}y_{3}-h$ is an eigenvalue of $J_{E_{3}}$. By assuming
$cx_{3}y_{3}-h= h\left(R^{CTL,W}_{2}-1\right)$,
we deduce that the sign of this
eigenvalue is negative when $R^{CTL,W}_{2}<1$, zero when
$R^{CTL,W}_{2}=1$, and positive for $R^{CTL,W}_{2}>1$. On the other
hand, from the Routh--Hurwitz theorem, the other eigenvalues of the
above matrix have negative real parts when $R^{W}>1$. Consequently,
$E_{3}$ is unstable when $R^{CTL,W}_{2}>1$ and locally
asymptotically stable when $R^{W}>1$ and $R^{CTL,W}_{2}<1$.
\end{proof}

For the last endemic-equilibrium point $E_4$, we prove the following result.

\begin{proposition}\
\label{prop4}
\begin{enumerate}
\item If $R^{CTL,W}_{1}<1$ or $R^{CTL,W}_{2} < 1$, then
the point $E_{4}$ does not exists. Moreover, $E_{4}=E_{3}$ when
$R^{CTL,W}_{2}=1$ and $E_{4}=E_{2}$ when $R^{CTL,W}_{1}=1$.
\item If $R^{CTL,W}_{1}> 1$ and $R^{CTL,W}_{2}> 1$,
then $E_{4}$ is locally asymptotically stable.
\end{enumerate}
\end{proposition}

\begin{proof}
It is clear that when $R^{CTL,W}_{1}<1$ or $R^{CTL,W}_{2} < 1$ the
point $E_{4}$ does not exists and, if $R^{CTL,W}_{2}=1$, then
$E_{4}=E_{3}$ and $E_{4}=E_{2}$ when $R^{CTL,W}_{1}=1$. We assume
that $R^{CTL,W}_{1}> 1$ and $R^{CTL,W}_{2}> 1$. The Jacobian matrix
of the system at the point $E_{4} $ is given by
\begin{equation}
\label{jb4}
J_{E_{4}}=\left(
\begin{array}{clcrc}
-d- \beta v_{4} & 0 & - \beta x_{4} & 0 & 0  \\
\beta v_{4}& -a-pz_{4} & \beta x_{4} & -py_{4} & 0 \\
0 & aN & -\mu-qw_{4} &0 & -qv_{4}\\
cy_{4}z_{4} & cx_{4}z_{4} & 0 & cx_{4}y_{4}-h & 0\\
gv_{4}w_{4} & 0 & gx_{4}w_{4} & 0 & gx_{4}v_{4}- \alpha
\end{array}
\right).
\end{equation}
The characteristic equation associated with $J_{E_4}$ is given by
$$
\xi^5+{\mathcal A}_{4}\xi^4+{\mathcal B}_{4}\xi^3
+{\mathcal C}_{4}\xi^{2}+{\mathcal D}_{4} \xi + {\mathcal E}_{4}=0,
$$
where
\begin{equation*}
\begin{aligned}
{\mathcal A}_{4}&= a+d+ \mu + \beta v_{4}+ pz_{4}+qw_{4},& \\
{\mathcal B}_{4}&=  (d+ \beta v_{4})(a+ \mu)+a \mu +pz_{4}(d+h
+ \mu + \beta v_{4}+ qw_{4})&\\
&\quad +qw_{4}(d+a+ \alpha+ \beta v_{4})-aN \beta x_{4},&\\
{\mathcal C}_{4}&= a \mu (d+ \beta v_{4})+pz_{4}(d \mu + dh +\mu h
+\mu \beta v_{4}+ h \beta v_{4})&\\
&\quad +qw_{4}(ad+ \alpha d+a \alpha +a \beta v_{4})
+pqz_{4}w_{4}(d+ \alpha +h+ \beta v_{4})-aN \beta dx_{4} ,&\\
{\mathcal D}_{4}&=  ad qw_{4} +pz_{4}(dh \mu +\mu h \beta v_{4}
-aN \beta h y_{4})+pqz_{4}w_{4}(d \alpha+ \alpha \beta v_{4}+h \alpha),&\\
{\mathcal E}_{4}&=  \alpha hd(pqz_{4}w_{4}+ aN \beta v_{4}-aN \beta x_{4}).&\\
\end{aligned}
\end{equation*}
From the Routh--Hurwitz theorem applied to the fifth order
polynomial, the eigenvalues of the Jacobian matrix \eqref{jb4}
have negative real parts since we have ${\mathcal A}_{4}>0$,
${\mathcal A}_{4}{\mathcal B}_{4}>{\mathcal C}_{4}$, ${\mathcal A}_{4}
{\mathcal B}_{4}{\mathcal C}_{4}>{\mathcal A}_{4}^{2}{\mathcal D}_{4}$,
and ${\mathcal A}_{4}{\mathcal B}_{4}{\mathcal C}_{4}{\mathcal D}_{4}
>{\mathcal A}_{4}{\mathcal B}_{4}^{2}{\mathcal E}_{4}
+{\mathcal A}_{4}^{2}{\mathcal D}_{4}^{2}$. Consequently, we obtain
the asymptotic local stability of the endemic point $E_{4}$.
\end{proof}


\section{Optimal control}
\label{sec:4}

In this section, we study an optimal control problem by introducing
drug therapy into the model \eqref{sy} and assuming that treatment
reduces the viral replication. Our purpose is to find a treatment
strategy $u(t)$ that maximizes the number of CD4$^+$ T-cells as well
as the number of CTL and antibody immune response, keeping the cost,
measured in terms of chemotherapy strength and a combination of
duration and intensity, as low as possible.


\subsection{The optimization problem}
\label{sec:4:1}

To apply optimal control theory, we suggest the following
control system with two control variables:
\begin{equation}
\label{sy1}
\left\{
\begin{array}{llll}
\vspace{0.1cm} \displaystyle \frac{dx(t)}{dt}
= \lambda -dx(t)- \beta(1-u_{1}(t)) x(t)v(t), \\
\vspace{0.1cm} \displaystyle \frac{dy(t)}{dt}
= \beta (1-u_{1}(t)) x(t)v(t)- ay(t)- py(t)z(t),\\
\vspace{0.1cm} \displaystyle \frac{dv(t)}{dt}
= aN(1-u_{2}(t))y(t)- \mu v(t)-qv(t)w(t) , \\
\vspace{0.1cm} \displaystyle \frac{dz(t)}{dt}
= cx(t)y(t)z(t) - hz(t),\\
\vspace{0.1cm} \displaystyle \frac{dw(t)}{dt}
= gx(t)v(t)w(t)- \alpha w(t).\\
\end{array}
\right.
\end{equation}
Here, $u_1$ represents the efficiency of drug therapy in blocking
new infection, so that infection rate in the presence of drug is
$(1 - u_1)$; while $u_2$ stands for the efficiency of drug therapy
in inhibiting viral production, such that the virion production
rate under therapy is $(1 - u_2)$.
Our optimization problem consists to maximize the
following objective functional:
\begin{equation}
\label{sy2}
\begin{aligned}
J(u_{1},u_{2})=\int^{t_{f}}_{0}\left\{x(t)+z(t)+w(t)
-\left[\frac{A_{1}}{2}u_{1}^{2}(t)+\frac{A_{2}}{2}u_{2}^{2}(t)\right]\right\}dt,
\end{aligned}
\end{equation}
where $t_{f}$ is the time period of treatment and the positive
constants $A_{1}$ and $A_{2}$ stand for the costs of
the introduced treatment. The two control functions, $u_{1}$
and $u_{2}$, are assumed to be bounded and Lebesgue integrable.
We look for $u_{1}^{*}$ and $u_{2}^{*}$ such that
\begin{equation}
\label{sy3}
\begin{aligned}
J(u_{1}^{*},u_{2}^{*})
=\max\left\{J(u_{1},u_{2}) : (u_{1},u_{2})\in U\right\},
\end{aligned}
\end{equation}
where $U$ is the control set defined by
\begin{equation*}
\begin{aligned}
U=\left\{(u_{1}(\cdot),u_{2}(\cdot)): u_{i}(\cdot)
\  \text{is measurable}, \, 0
\leq u_{i}(t)\leq 1,\, t\in [0,t_{f}],\, i=1,2\right\}.
\end{aligned}
\end{equation*}
Note that it is natural to maximize the number of CTL
and immune response in the optimal control problem.
Indeed, it has been noted clinically
that individuals who maintain a high level of CTLs remain healthy longer.
Therefore, we wish to maximize the number of CTL so as to ensure that
if viral load does rebound, the immune system will be able to handle it.
The best drug treatments should establish this result,
while keeping adverse effects to a minimum.


\subsection{Existence of an optimal control pair}
\label{sec:4:2}

The existence of the optimal control pair can be directly obtained
using the results in \cite{Fleming,Lukes}. More precisely, we have
the following theorem.

\begin{theorem}
\label{thm:4.1}
There exists an optimal control pair $(u_{1}^{*},u_{2}^{*})\in U$
solution of \eqref{sy1}--\eqref{sy3}.
\end{theorem}

\begin{proof}
To use the existence result in \cite{Fleming}, we first need to
check the following properties:
\begin{enumerate}
\item [$(P_1)$] the set of controls and corresponding state
variables is nonempty;

\item [$(P_2)$] the control set $U$ is convex and closed;

\item [$(P_3)$] the right-hand side of the state system is bounded
by a linear function in the state and control variables;

\item [$(P_4)$] the integrand of the objective functional is
concave on $U$;

\item [$(P_5)$] there exist constants $c_{1}, c_{2} > 0$ and
$\beta > 1$ such that the integrand
$$
L(x,z,w,u_{1},u_{2}) = x + z + w - \left( \frac{A_{1}}{2}u_{1}^{2}
+\frac{A_{2}}{2}u_{2}^{2}\right)
$$
of the objective functional \eqref{sy2} satisfies
\begin{equation*}
L(x,z,w,u_{1},u_{2}) \leq c_{2}-c_{1}(\mid u_{1}\mid^{2} + \mid
u_{2}\mid^{2})^{^{\frac{\beta}{2}}}.
\end{equation*}
\end{enumerate}
Using the result in \cite{Lukes}, we obtain existence of solutions
of system \eqref{sy1}, which gives condition $(P_1)$. The control
set is convex and closed by definition, which gives condition
$(P_2)$. Since our state system is bilinear in $u_{1}$ and $u_{2}$,
the right-hand side of system \eqref{sy1} satisfies condition
$(P_3)$, using the boundedness of solutions. Note that the integrand
of our objective functional is concave. Also, we have the last
needed condition:
\begin{equation*}
L(x,z,w,u_{1},u_{2}) \leq c_{2} -c_{1}\left(\mid u_{1}\mid^{2}
+ \mid u_{2}\mid^{2}\right),
\end{equation*}
where $c_{2}$ depends on the upper bound on $x$, and $c_{1}>0$ since
$A_{1}>0$, $A_{2}>0$. We conclude that there exists an optimal
control pair $(u_{1}^{*},u_{2}^{*})\in U$ such that
$J(u_{1}^{*},u_{2}^{*}) = \displaystyle\max_{(u_{1},u_{2})\in U}
\mathcal{J}(u_{1},u_{2})$. 
\end{proof}

Theorem~\ref{thm:4.1} does not provide a uniqueness 
result for the optimal control problem. The uniqueness of the optimal controls 
is obtained in terms of the unique solution of the optimality system.


\subsection{The optimality system}
\label{sec:4:3}

Pontryagin's minimum principle provides necessary optimality conditions for
such optimal control problem \cite{PontryaginMP}. This principle
transforms (\ref{sy1}), (\ref{sy2}) and (\ref{sy3}) into a problem
of minimizing an Hamiltonian, $H$, pointwisely with respect to
$u_{1}$ and $u_{2}$, where
$$
H(t,x,y,v,z,w,u_{1},u_{2},\lambda)=\frac{A_{1}}{2}u_{1}^{2}+
\frac{A_{2}}{2}u_{2}^{2}-x-z-w+\displaystyle\sum_{i=0}^{5}\lambda_{i}f_{i}
$$
with
\begin{equation*}
\left\{
\begin{aligned}
f_{1} &= \lambda-dx- \beta(1-u_{1})xv, \\
f_{2}&= \beta(1-u_{1})xv-ay-pyz, \\
f_{3}&= aN(1-u_{2})y- \mu v-qvw, \\
f_{4}&= cxyz-hz,\\
f_{5}&= gxvw- \alpha w.
\end{aligned}
\right.
\end{equation*}
By applying Pontryagin's minimum principle \cite{PontryaginMP}, we
obtain the following result.

\begin{theorem}
Given optimal controls $u_{1}^{*}$, $u_{2}^{*}$, and solutions $x^{*}$,
$y^{*}$, $v^{*}$, $z^{*}$, and $w^{*}$ of the corresponding state system
\eqref{sy1}, there exists adjoint variables $\lambda_{1}$,
$\lambda_{2}$, $\lambda_{3}$, $\lambda_{4}$, and $\lambda_{5}$
satisfying the equations
\begin{equation}
\label{eq:adj:syst}
\left\{
\begin{array}{ll}
\lambda'_{1}(t)
&=1+\lambda_{1}(t)\big[d+\big(1-u_{1}^{*}(t)\big)\beta v^{*}(t)\big]
-\lambda_{2}(t)(1-u_{1}^{*}(t))\beta v^{*}(t)\\
&\quad - \lambda_{4}(t)cy^{*}(t)z^{*}(t) - \lambda_{5}(t) gv^{*}(t)w^{*}(t),\\
\lambda'_{2}(t)&=\lambda_{2}(t)(a+pz^{*}(t))-\lambda_{3}(t)\big(1-u^{*}_{2}(t)\big)aN
- \lambda_{4}(t)cx^{*}(t)z^{*}(t),\\
\lambda'_{3}(t)&=\lambda_{1}(t)(1-u_{1}^{*}(t))\beta x^{*}(t)
-\lambda_{2}(t)(1-u_{1}^{*}(t)\big)\beta x^{*}(t)+\lambda_{3}(t)(\mu+qw^{*}(t))\\
&\quad -\lambda_{5}(t)gx^{*}(t)w^{*}(t),\\
\lambda'_{4}(t)&=1+\lambda_{2}(t)py^{*}(t)+\lambda_{4}(t)\left[h-cx^{*}(t)y^{*}(t)\right],\\
\lambda'_{5}(t)&=1+\lambda_{3}(t)qv^{*}(t)+\lambda_{5}(t)\left[\alpha-gx^{*}(t)v^{*}(t)\right],\\
\end{array}\right.
\end{equation}
with the transversality conditions
\begin{equation}
\label{eq:TC}
\lambda_{i}(t_{f})=0, \quad i=1,\ldots,5.
\end{equation}
Moreover, the optimal control is given by
\begin{equation}
\label{opt}
\begin{aligned}
u^{*}_{1}(t)&= \min\left(1,\max\left(0,\frac{\beta}{A_{1}}\bigg[
(\lambda_{2}(t)-\lambda_{1}(t)) x^{*}(t)v^{*}(t)\bigg]\right)\right),\\
u^{*}_{2}(t)&= \min\left(1,\max\left(0,\frac{1}{A_{2}}\lambda_{3}(t)aNy^{*}(t)\right)\right).
\end{aligned}
\end{equation}
\end{theorem}

\begin{proof}
The proof of positivity and boundedness of solutions is similar to
the one of Proposition~\ref{prop1}. It is enough to use the fact
that $u_i(\cdot) \in U$, $i=1,2$, which means that
$\|u_i(\cdot)\|_{L^{\infty}} \le 1$. For the rest of the proof,
we remark that the adjoint equations and transversality conditions
are obtained by using the Pontryagin minimum principle
of \cite{PontryaginMP}, from which
\begin{equation*}
\begin{cases}
\lambda'_{1}(t)=-\displaystyle \frac{\partial H}{\partial x},
\qquad &\lambda_{1}(t_{f})=0,\\[0.3cm]
\lambda'_{2}(t)=-\displaystyle \frac{\partial H}{\partial y},
\qquad &\lambda_{2}(t_{f})=0,\\[0.3cm]
\lambda'_{3}(t)=-\displaystyle \frac{\partial H}{\partial v},
\qquad &\lambda_{3}(t_{f})=0,\\[0.3cm]
\lambda'_{4}(t)=-\displaystyle \frac{\partial H}{\partial z},
\qquad &\lambda_{4}(t_{f})=0,\\[0.3cm]
\lambda'_{5}(t)=-\displaystyle \frac{\partial H}{\partial w},
\qquad &\lambda_{5}(t_{f})=0.
\end{cases}
\end{equation*}
From the optimality conditions
$$
\displaystyle \frac{\partial H}{\partial u_{1}}=0
\quad \text{ and } \quad
\displaystyle \frac{\partial H}{\partial u_{2}}=0,
$$
that is,
\begin{gather*}
A_{1}u_{1}(t)+ \lambda_{1}(t) \beta x^{*}(t)v^{*}(t)
-\lambda_{2}(t) \beta x^{*}(t)v^{*}(t)=0,\\
A_{2}u_{2}(t)-aNy^{*}(t) \lambda_{3}(t)= 0,
\end{gather*}
and taking into account the bounds in $U$ for the two controls, one
obtains $u_{1}^{*}$ and $u_{2}^{*}$ in the form \eqref{opt}.
\end{proof}

The optimality system consists of the state system \eqref{sy1}
coupled with the adjoint equations \eqref{eq:adj:syst},
the initial conditions \eqref{eq:IC}, transversality
conditions \eqref{eq:TC}, and the characterization of optimal controls
\eqref{opt}. Precisely, if we substitute the expressions of
$u_{1}^{*}$ and $u_{2}^{*}$ in \eqref{sy1}, then we obtain the
following optimality system:
\begin{equation}
\label{eq:optSyst}
\left\{
\begin{split}
\displaystyle \frac{dx^{*}(t)}{dt}&= \lambda- dx^{*}(t)
-\beta(1-u^{*}_{1}(t)) x^{*}(t) v^{*}(t),\\[0.3cm]
\displaystyle \frac{dy^{*}(t)}{dt}&= \beta(1-u^{*}_{1}(t))
x^{*}(t) v^{*}(t)-ay^{*}(t)-py^{*}(t)z^{*}(t),\\[0.3cm]
\displaystyle \frac{dv^{*}(t)}{dt}&= aN(1-u_{2}^{*}(t))y^{*}(t)
-\mu v^{*}(t)-qv^{*}(t)w^{*}(t),\\[0.3cm]
\displaystyle \frac{dz^{*}(t)}{dt}
&=cx^{*}(t)y^{*}(t)z^{*}(t)-hz^{*}(t),\\[0.3cm]
\displaystyle \frac{dw^{*}(t)}{dt}
&=gx^{*}(t)v^{*}(t)w^{*}(t) -\alpha w^{*}(t),\\[0.3cm]
\displaystyle \frac{d\lambda_{1}(t)}{dt}
&=1+\lambda_{1}(t)\big[d+\big(1-u_{1}^{*}(t)\big)\beta v^{*}(t)\big]
-\lambda_{2}(t)(1-u_{1}^{*}(t)\big)\beta v^{*}(t)\\[0.3cm]
&\qquad - \lambda_{4}(t)cy^{*}(t)z^{*}(t)
- \lambda_{5}(t)g v^{*}(t)w^{*}(t),\\[0.3cm]
\displaystyle \frac{d\lambda_{2}(t)}{dt}
&=\lambda_{2}(t)(a+pz^{*}(t))-\lambda_{3}(t)\big(1-u^{*}_{2}(t)\big)aN
- \lambda_{4}(t)cx^{*}(t)z^{*}(t),\\[0.3cm]
\displaystyle \frac{d\lambda_{3}(t)}{dt}
&=\lambda_{1}(t)(1-u_{1}^{*}(t))\beta
x^{*}(t)-\lambda_{2}(t)(1-u_{1}^{*}(t)\big)\beta x^{*}(t)\\[0.3cm]
&\qquad +\lambda_{3}(t)(\mu+qw^{*}(t))
-\lambda_{5}(t)gx^{*}(t)w^{*}(t),\\[0.3cm]
\displaystyle \frac{d\lambda_{4}(t)}{dt}
&=1+\lambda_{2}(t)py^{*}(t)
+\lambda_{4}(t)\big[h-cx^{*}(t)y^{*}(t)\big],\\[0.3cm]
\displaystyle \frac{d\lambda_{5}(t)}{dt}
&=1+\lambda_{3}(t)qv^{*}(t)+\lambda_{5}(t)\big[\alpha-gx^{*}(t)v^{*}(t)\big],\\[0.3cm]
u^{*}_{1}&=\min\bigg(1,\max\bigg(0,\frac{\beta}{A_{1}}\bigg[(\lambda_{2}(t)
-\lambda_{1}(t)) x^{*}(t)v^{*}(t)\bigg]\bigg)\bigg),\\[0.3cm]
u^{*}_{2}&=\min\bigg(1,\max\bigg(0,\frac{1}{A_{2}}\lambda_{3}(t)
aNy^{*}(t)\bigg)\bigg),\\[0.3cm]
\lambda_{i}(t_{f})&=0, \quad i=1,\ldots,5.
\end{split}\right.
\end{equation}


\section{Numerical simulations}
\label{sec:5}

In order to solve the optimality system \eqref{eq:optSyst},
we use a numerical scheme based on forward
and backward finite difference approximations. Precisely, we
implemented Algorithm~\ref{our:alg}.

\begin{algorithm}
\caption{Numerical algorithm for the optimal control problem
\eqref{sy1}--\eqref{sy3}.}\label{our:alg}
\flushleft
\medskip

\underline{Step 1}:

\medskip
$$
x (0) = x_0, \quad y(0) = y_0,
\quad v(0) = v_0,
\quad z(0) = z_0,
\quad w(0) = w_0,
\quad u_1(0) = 0,
$$
$$
u_2(0) = 0,\quad
\lambda_1(t_f) = 0, \quad \lambda_2(t_f) = 0,
\quad \lambda_3(t_f) = 0,
\quad \lambda_4(t_f) =0, \quad \lambda_5(t_f) = 0.
$$

\medskip

\underline{Step 2}:

\medskip

for $i = 0$, \ldots , $n-1$, do:
\begin{equation*}
\begin{split}
x_{i+1} &= x_i + h[\lambda - d x_i - \beta(1-u_{1}^i)  x_i v_i],\\
y_{i+1} &= y_i + h[ \beta (1-u_{1}^i) x_{i-m}v_{i-m}-ay_i-py_i z_i],\\
v_{i+1} &= v_i + h[ aN(1-u_2^i)y_i-\mu v_i-q v_i w_i],\\
z_{i+1} &= z_i + h[ cx_i y_i z_i-hz_i],\\
w_{i+1} &= w_i + h[ gx_iv_i w_i - \alpha w_i],\\
\lambda_1^{n-i-1}
&=  \lambda_1^{n-i} - h[ 1 + \lambda_1^{n-i}(d
+(1-u_{1}^i)\beta v_{i+1})\\
&\quad - \lambda_2^{n-i}(1-u_{1}^i)\beta v_{i+1} - \lambda_4^{n-i}
c y_{i+1} z_{i+1} - \lambda_5^{n-i} g v_{i+1} w_{i+1}],\\
\lambda_2^{n-i-1} &=\lambda_2^{n-i}-h[\lambda_2^{n-i}(a+pz_{i+1})\\
&\quad -\lambda_3^{n-i}(aN(1-u_{2}^i))- \lambda_4^{n-i} c x_{i+1} z_{i+1}],\\
\lambda_3^{n-i-1} &= \lambda_3^{n-i}-h\big[\lambda_1^{n-i}(1-u_1^i)\beta
x_{i+1}-\lambda_2^{n-i}(1-u_{1}^i)\beta x_{i+1}\\
&\quad + \lambda_3^{n-i}( \mu +qw_{i+1})-\lambda_5^{n-i} gx_{i+1}w_{i+1}],\\
\lambda_4^{n-i-1} &= \lambda_4^{n-i} - h[1+ \lambda_2^{n-i}p y_{i+1}
+ \lambda_4^{n-i}(h-cx_{i+1}y_{i+1})],\\
\lambda_5^{n-i-1} &= \lambda_5^{n-i} - h[1+ \lambda_3^{n-i}q v_{i+1}
+ \lambda_5^{n-i}(\alpha-gx_{i+1}v_{i+1})],\\
R_1^{i+1} &= (\beta/A_1)(\lambda_2^{n-i-1}v_{i-m+1}x_{i-m+1}
-\lambda_1^{n-i-1}v_{i+1}x_{i+1}),\\
R_2^{i+1} &= (1/A_2)\lambda_3^{n-i-1}aN y_{i+1},\\
u_1^{i+1} &=\min(1,\max(R_1^{i+1},0)),\\
u_2^{i+1} &=\min(1,\max(R_2^{i+1},0)),
\end{split}
\end{equation*}
end for.

\medskip

\underline{Step 3}:

\medskip

for $i = 1, \ldots , n$, write:
$$
x^*(t_i) = x_i, \quad y^*(t_i) = y_i, \quad v^*(t_i) =v_i, \quad
z^*(t_i) = z_i, \quad w^*(t_i) = w_i,$$ $$ \quad u_1^*(t_i) = u_1^i,
\quad u_2^*(t_i) = u_2^i,
$$

end for.
\end{algorithm}

{\small
\begin{table}
\caption{Parameters, their symbols and meaning, and default values
used in HIV literature.}
\begin{tabular}{c c c c} \hline
Parameters & \multicolumn{1}{c}{Meaning} &  Value &  References \\
\hline \hline $\lambda$ &  \multicolumn{1}{p{4cm}}{\raggedright
source rate of CD4+ T cells}
& $1$--$10$ cells $\mu l^{-1}$ days$^{-1}$ & \cite{crs}\\
$d$ & \multicolumn{1}{p{4cm}}{decay rate of healthy cells}
& $0.007$--$0.1$ days$^{-1}$ & \cite{crs} \\
$\beta$   & \multicolumn{1}{p{4cm}}{rate at which CD4+ T cells
become infected}
& $0.00025$--$0.5$ $\mu l$ virion$^{-1}$ days$^{-1}$   & \cite{crs}\\
$a$ & \multicolumn{1}{p{4cm}}{death rate of infected CD4+ T cells,
not by CTL}
& $0.2$--$0.3$ days$^{-1}$ & \cite{crs} \\
$\mu$  & \multicolumn{1}{p{4cm}}{clearance rate of virus}
& $2.06$--$3.81$ days$^{-1}$ & \cite{per}\\
$N$ & \multicolumn{1}{p{4cm}}{number of virions produced by infected
CD4+ T-cells}
& $6.25$--$23599.9$ virion$^{-1}$ & \cite{5a,32}\\
$p$ & \multicolumn{1}{p{4cm}}{clearance rate of infection}
& $1$--$4.048 \times 10^{-4}$ ml virion days$^{-1}$ & \cite{5a,22} \\
$c$ & \multicolumn{1}{p{4cm}}{activation rate of CTL cells}
& $0.0051$--$3.912$ days$^{-1}$ & \cite{5a} \\
$h$ & \multicolumn{1}{p{4cm}}{death rate of CTL cells}
& $0.004$--$8.087$ days$^{-1}$ & \cite{5a} \\
$q$ & \multicolumn{1}{p{4.5cm}}{Neutralization rate of virions}
& $0.12$ days$^{-1}$ & Assumed \\
$g$ & \multicolumn{1}{p{4.3cm}}{activation rate of antibodies}
& $0.00013$ days$^{-1}$ & Assumed \\
$ \alpha$ & \multicolumn{1}{p{4cm}}{death rate of antibodies} &
$0.12$ days$^{-1}$ & Assumed \\ \hline
\end{tabular}
\label{table:nonlin}
\end{table}}

For our numerical simulations, we have chosen the following
parameters (see Table $1$): $\lambda=1$, $d=0.1$,  $\beta=0.00025$,
$p=.001$,  $a=0.2$, $c=0.03$, $N=2000$, $\mu=2.4$, $h=0.2$,
$g=0.00013$, $\alpha=0.12$, $q=0.01$, $A_1=250$, $A_2=2500$.
These parameters show the stability of the last endemic
point $E_4$ with all non-zero system components. The initial
value of each system component is given as follows:
$x_0=5$, $y_0=1$, $v_0=1$, $z_0=2$, and $w_0=1$.
In Fig.~\ref{fig1}, it can be clearly seen that, after introducing
therapy, the uninfected cells population grows significantly
compared with those without control.
\begin{figure}[H]
\centering
\includegraphics[scale=0.34]{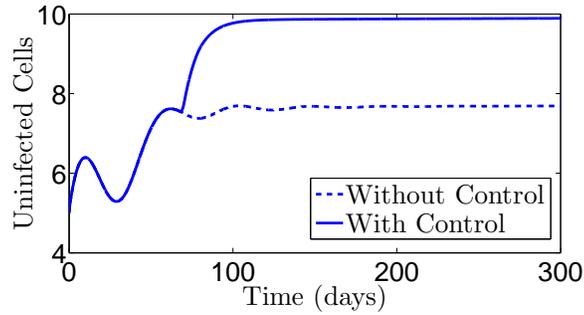}
\caption{The evolution of the uninfected cells during
time.}\label{fig1}
\end{figure}
Figure~\ref{fig2} shows that,
with control, the number of infected cells are significantly reduced
after few weeks of therapy. Nevertheless, without control, this
number remains much higher.
\begin{figure}[H]
\centering
\includegraphics[scale=0.34]{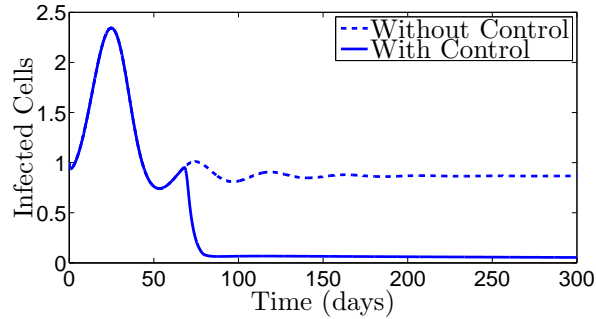}
\caption{The evolution of the infected cells during time.}
\label{fig2}
\end{figure}
Figure~\ref{fig3} shows that, with
control, the viral load decreases towards a very low level after the
first days of therapy, whereas, without control, it remains much
higher. This indicates the impact of the administrated therapy in
controlling viral replication.
\begin{figure}[H]
\centering
\includegraphics[scale=0.34]{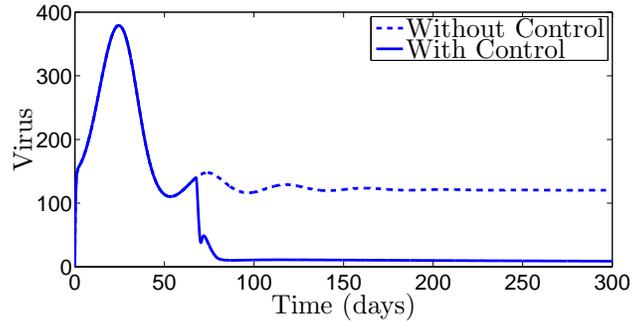}
\caption{The evolution of the HIV virus during time.}\label{fig3}
\end{figure}
Figures~\ref{fig4} and \ref{fig5}
show the adaptive immune response as function of time. The adaptive
immunity is clearly affected by the control. Their curves converge
towards zero with control, whereas, without any control, it
converges towards $66.2721$ for CTL cells and converge towards
$48.888$ for antibodies immune response.
\begin{figure}[H]
\centering
\includegraphics[scale=0.34]{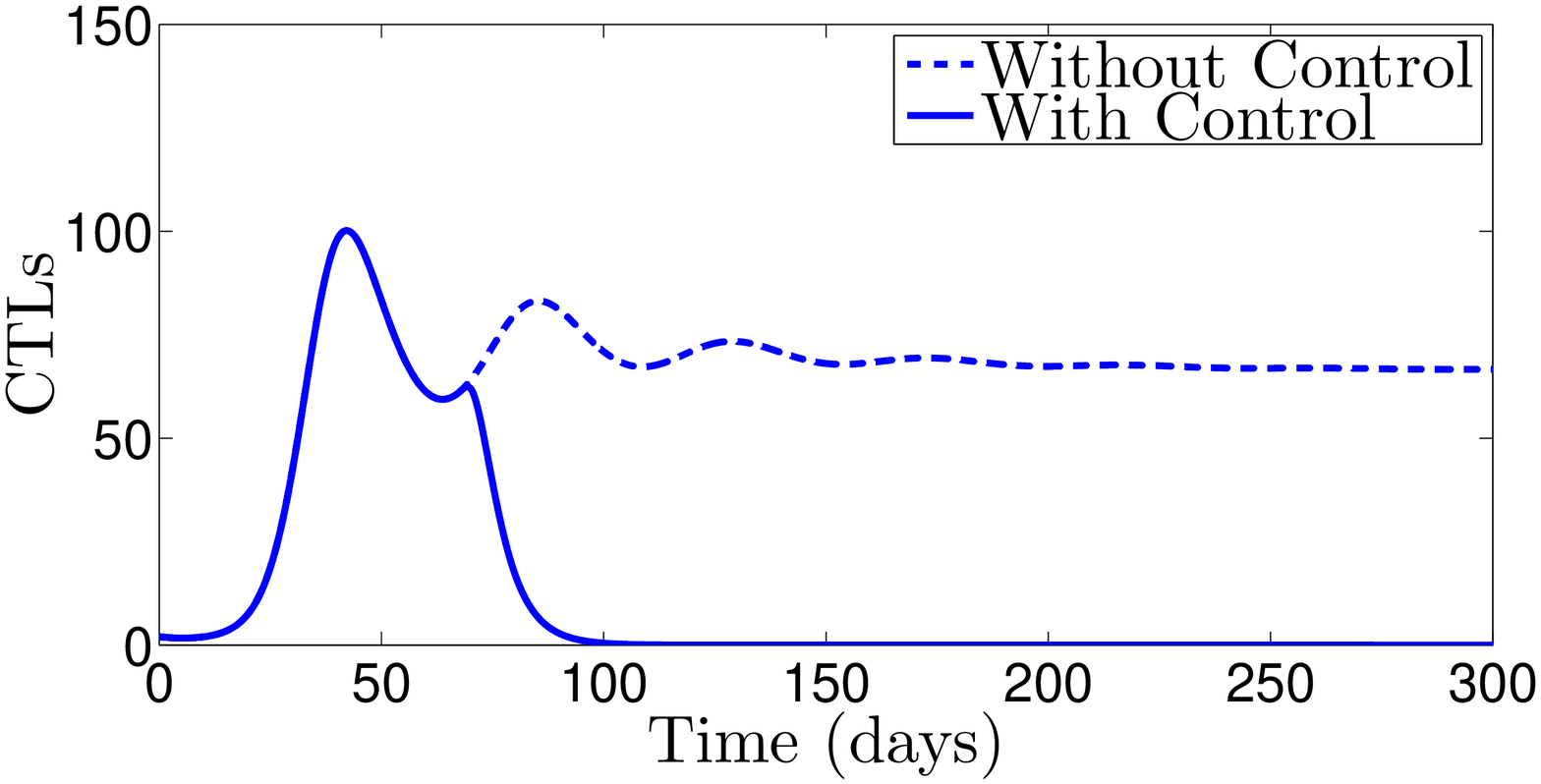}
\caption{The evolution of the CTL cells during time.}\label{fig4}
\end{figure}
\begin{figure}[H]
\centering
\includegraphics[scale=0.34]{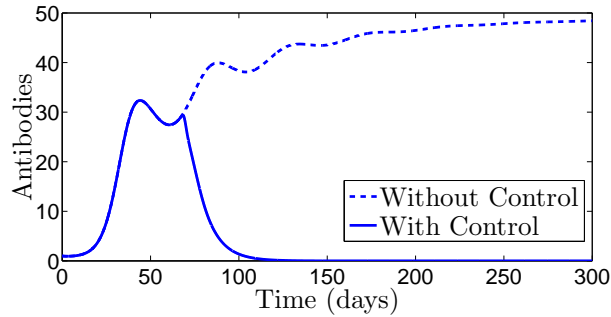}
\caption{The evolution of the antibodies during time.}\label{fig5}
\end{figure}
We note that all the curves
(without control) of previous figures converge
towards the endemic point with coordinates
$(7.6923,0.8666,120,66.2721,48.888)$. This result is in good
agreement with the result of  Proposition~\ref{prop4}, since with our
chosen parameters we have $R_0 =2.0833 > 1$, $R^{CTL,W}_{1} =
1.2037 > 1$, and $R^{CTL,W}_{2} = 1.3313 > 1$. The behavior of the
two treatments during time is given in Fig.~\ref{fig6}. The curves
present the drug administration schedule during the time of
treatment. This figure shows that we should give more importance to
the first drug (RTIs) than to the second one (PIs).
\begin{figure}[H]
\centering
\includegraphics[scale=0.31]{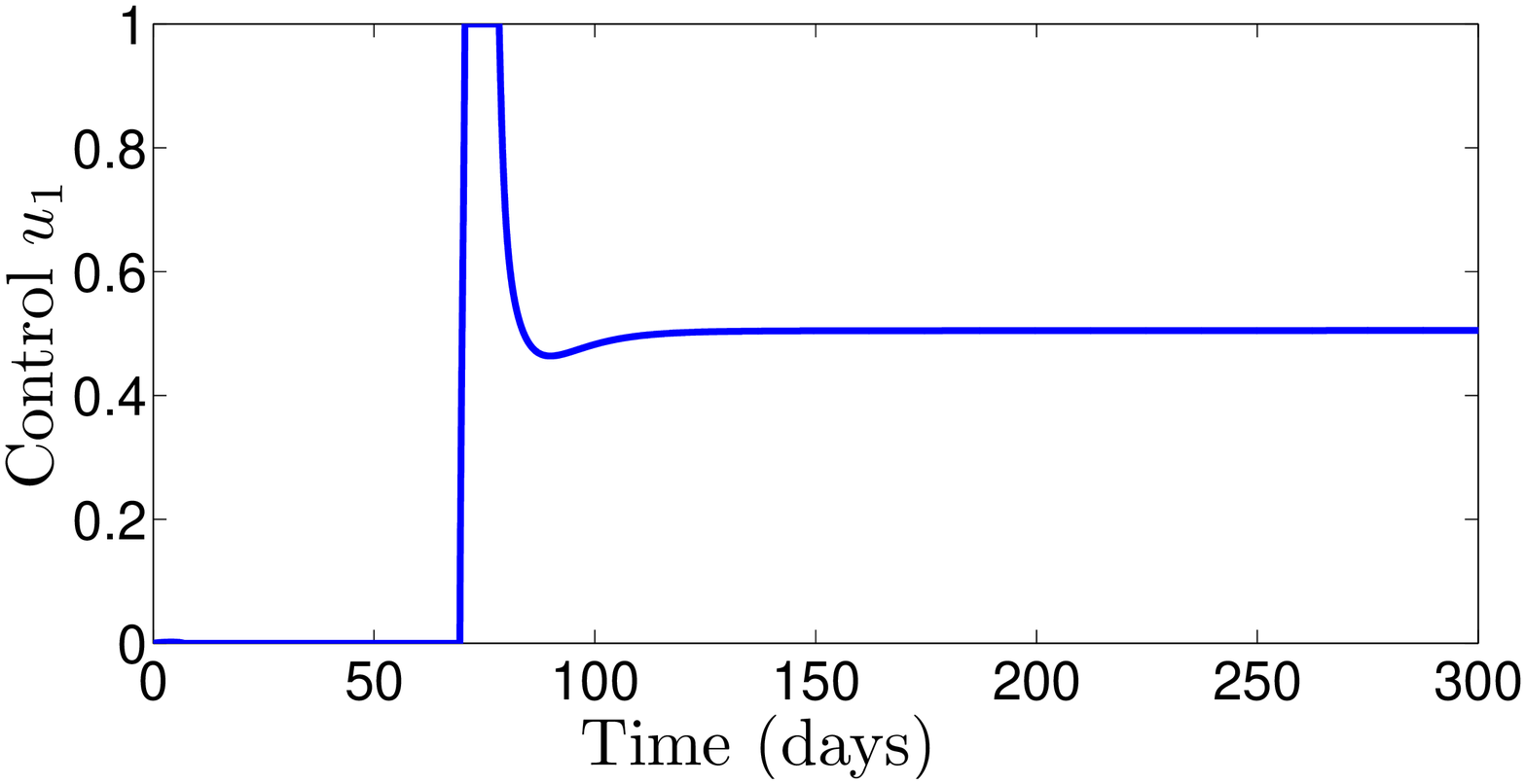}
\includegraphics[scale=0.31]{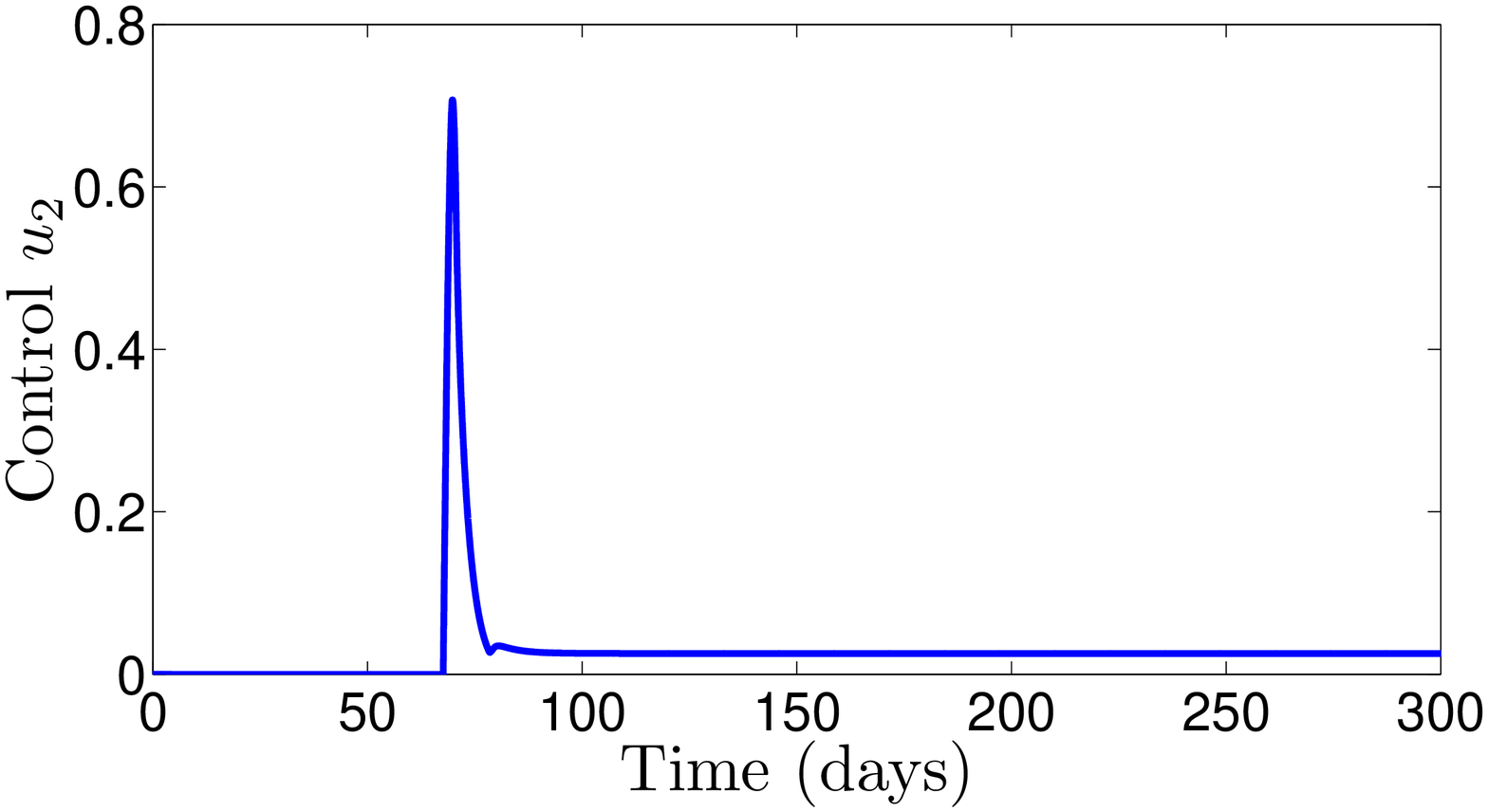}
\caption{The behaviour of the two optimal controls.}\label{fig6}
\end{figure}


\section{Conclusion}
\label{sec:6}

In this work, we proposed and studied a mathematical model describing the
human immunodeficiency virus with adaptive immune response and a
trilinear antibody growth function. The main novelty in the model
is to consider that the antibody growth depends not only on the virus
and on the antibodies concentration but also on the uninfected cells
concentration, which is supported by recent medical discoveries.
After proposing the new mathematical model, positivity and boundedness
of solutions were established. Then, local stability of
the disease free steady state and the infection steady states was
investigated. Next, an optimal control problem was proposed and
studied. Two types of treatments were incorporated into the
model: the purpose of the first consists to block the viral
proliferation, while the role of the second one is to prevent new
infections. Finally, numerical simulations were performed, confirming
the stability of the free and endemic equilibria and illustrating
the effectiveness of the two incorporated treatments via optimal
control.


\section*{Acknowledgments}

KA and SH would like to thank the Moroccan CNRST ``Centre National
de Recherche Scientique et Technique'' and the French CNRS ``Centre
National de Recherche Scientique'' for the partial financial support
through the project PICS 244832. DFMT is supported by
Funda\c{c}\~ao para a Ci\^encia e a Tecnologia (FCT),
within project UIDB/04106/2020 (CIDMA).
The authors are very grateful to an anonymous referee
for their critical remarks and precious questions, which
helped them to improve the quality and clarity of the manuscript.



\medskip

Received 01-June-2019; revised 10-May-2021; accepted 21-May-2021.

\medskip


\end{document}